\documentclass{amsart}
\usepackage[breaklinks]{hyperref}
\usepackage{fullpage}

\newcommand{\R}{\mathbb{R}}
\newcommand{\C}{\mathbb{C}}
\newcommand{\Z}{\mathbb{Z}}
\newcommand{\N}{\mathbb{N}}
\newcommand{\J}{\mathcal{J}}
\newcommand{\X}{X_{\Ha}}
\newcommand{\wind}{\operatorname{wind}}
\newcommand{\A}{\mathbf{A}}
\newcommand{\crit}{\operatorname{crit}}
\newcommand{\sign}{\operatorname{sign}}
\newcommand{\F}{\mathcal{F}}
\newcommand{\ind}{\operatorname{ind}}
\newcommand{\Ha}{\mathcal{H}}
\newcommand{\tl}{\tilde}

\newcommand{\End}{\operatorname{End}}

\def\br#1{\left\{#1\right\}}
\def\bp#1{\left(#1\right)}

\def\fl#1{\lfloor#1\rfloor}
\def\ceil#1{\lceil#1\rceil}
\def\abs#1{\left|#1\right|}

\newtheorem{lemma}{Lemma}[section]
\newtheorem{theorem}[lemma]{Theorem}
\newtheorem{corollary}[lemma]{Corollary}

\theoremstyle{definition}

\newtheorem*{ack}{Acknowledgements}

\newtheorem{assumptions}[lemma]{Assumptions}

\title[Holomorphic curves and holomorphic hyperplane foliations]{Holomorphic curves in the presence of holomorphic hypersurface foliations}
\author[A.\ Moreno]{Agustin Moreno}
\address{Institut f\"ur Mathematik \\ Universit\"at Augsburg  
\\  86159 Augsburg \\ Germany}
\email{\href{mailto:agustin.moreno@math.uni-augsburg.de}{agustin.moreno@math.uni-augsburg.de}}
\author[R.\ Siefring]{Richard Siefring}
\address{Fakult\"at f\"ur Mathematik \\ Ruhr-Universit\"at Bochum \\
     44780 Bochum \\
    Germany}
\email{\href{mailto:richard.siefring@ruhr-uni-bochum.de}{richard.siefring@ruhr-uni-bochum.de}}
\urladdr{\url{http://homepage.ruhr-uni-bochum.de/richard.siefring}}
\thanks{R.\ Siefring received support from DFG grant BR 5251/1-1 during the completion of this project.}
\date{February 7, 2019}

\begin{document}
\begin{abstract}
We prove a result which establishes restrictions on the pseudoholomorphic curves
which can exist in
a stable Hamiltonian manifold in the presence of certain
$\R$-invariant foliations of the symplectization by holomorphic hypersurfaces.
This result has applications in the first author's work \cite{Moreno:thesis, Moreno:torsion}
on algebraic torsion in higher dimensional contact manifolds.
\end{abstract}

\maketitle
\tableofcontents

\section{Background and main result}
Our main results here are motivated by the study of algebraic torsion in contact manifolds \cite{Moreno:thesis,Moreno:torsion}
and concern finding restrictions on the existence of pseudoholomorphic curves
in certain manifolds equipped with stable Hamiltonian structures.
The manifolds we consider will be smooth fibrations over a closed, oriented surface and
we will assume further that the symplectization admits an $\R$-invariant foliation by pseudoholomorphic
hypersurfaces which project to gradient flow lines of a Morse function on the surface.
We are interested in identifying conditions which will guarantee that a punctured pseudoholomorphic curve
is contained in the image of a leaf of a foliation.
Before stating the main results we give some definitions.

Let $M^{2n+1}$ be a closed, orientable manifold.  A pair
$\Ha=(\lambda, \omega)\in\Omega^{1}(M)\times\Omega^{2}(M)$ is said to be a
\emph{stable Hamiltonian structure} on $M$ if
\begin{itemize}
\item $\lambda\wedge \omega^{n}$ is a volume form on $M$,
\item $d\omega=0$, and
\item $d\lambda$ vanishes on the kernel of the map $v\mapsto i_{v}\omega$.
\end{itemize}
A stable Hamiltonian structure on $M$ determines a splitting
\[
TM=\R\X\oplus (\xi, \omega|_{\xi})
\]
of the tangent space of $M$ into a symplectic hyperplane distribution
$(\xi=\ker\lambda, \omega|_{\xi})$ and a line bundle determined by the span of the 
\emph{Reeb vector field} $\X$, which is the unique vector field satisfying
\[
\lambda(\X)=1 \qquad\text{ and }\qquad i_{\X}\omega=0.
\]
We will refer to the triple $(M, \lambda, \omega)$ as a \emph{stable Hamiltonian manifold}.
A stable Hamiltonian structure $(\lambda, \omega)$ on $M$ is said to be \emph{nondegenerate} if
all periodic orbits of the Reeb vector field are nondegenerate.

A codimension-$2$ submanifold $V\subset M$ is said to be a
\emph{stable Hamiltonian hypersurface} of $M$ if
the pair $\Ha':=(\lambda', \omega')$ defined by
\[
\lambda':=i^{*}\lambda \qquad \omega':=i^{*}\omega,
\]
where $i:V\hookrightarrow M$ is the inclusion map, is a stable Hamiltonian structure on $V$.
In this case, the hyperplane distribution $\xi':=\ker\lambda'$ is naturally identified via
$i_{*}$ with $TV\cap\xi$.
We say a stable Hamiltonian hypersurface $V\subset M$ is a
\emph{strong stable Hamiltonian hypersurface}
if, in addition, $V$ is invariant under the flow of $\X$.
This is easily seen to be equivalent to requiring that the push forward
by $i_{*}$ of the Reeb vector field $X_{\Ha'}$ of the stable Hamiltonian structure
$\Ha'$ is equal to $\X$ at all points in $V$.
Along a strong stable Hamiltonian hypersurface $V\subset M$, we thus
obtain a splitting of the tangent space of $M$
\[
TM|_{V}=\R\X\oplus (\xi', \omega')\oplus(\xi_{V}^{\perp}, \omega|_{\xi_{V}^{\perp}})
\]
into a line bundle spanned by $\X$ and two symplectic vector bundles,
where
\begin{equation}\label{e:symp-comp-def}
\xi_{V}^{\perp}=\br{v\in \xi|_{V}\,\middle|\,\omega(v, i_{*}w)=0 \quad\forall w\in \xi'},
\end{equation}
is the symplectic complement of $\xi'\approx TV\cap\xi$ in $\xi|_{V}$.
Moreover,
since the flow of $\X$ preserves $\lambda$ and $\omega$,
it also preserves this splitting.  Therefore, given a periodic orbit
$\gamma$ of $\X$ lying in $V$ and a symplectic trivialization $\Phi$ of
$\gamma^{*}\xi_{V}^{\perp}$, we can assign a normal Conley--Zehnder index
$\mu_{N}^{\Phi}(\gamma)$ by considering the restriction of the linearized flow along $\gamma$
to the symplectic normal bundle
$\xi_{V}^{\perp}$ of $\xi'$ in $\xi$.
We describe this construction in more detail in Section \ref{s:intersections} below.

We consider a manifold $M^{2n+1}$
equipped with a nondegenerate stable Hamiltonian structure $\Ha=(\lambda, \omega)$,
and we will denote by the triple $(\Sigma, p, Y^{2n-1})$
a smooth fibration $p:M\to\Sigma$
over a closed, oriented surface $\Sigma$ with fiber diffeomorphic to $Y$.
Given a Morse function $f$ on $\Sigma$ we say the fibration
$(\Sigma, p, Y)$ is \emph{$f$-admissible}
if for each critical point
$w\in\crit(f)$ of the function, the fiber
$Y_{w}=p^{-1}(w)$ over $w$
is a strong stable Hamiltonian hypersurface.
We will denote $f$-admissible fibrations by quadruples
$(\Sigma, p, Y, f)$.

Given an $f$-admissible fibration $(\Sigma, p, Y, f)$ for $(M, \Ha)$ and a critical point
$w\in\crit(f)$, we note that at any point $y\in Y_{w}$,
the derivative
$p_{*}$ at $y$ determines a linear isomorphism
\[
p_{*}(y):(\xi_{Y_{w}}^{\perp})_{y}\to T\Sigma_{w}
\]
from the symplectic normal bundle of $Y_{w}$ in $M$ at $y\in Y_{w}$
to the tangent space of $\Sigma$ at $w$.
This map will either be orientation preserving at every point in $Y_{w}$ or orientation 
reversing at every point in $Y_{w}$.
This allows us to define a sign function
\[
\sign:\crit(f)\to\br{-1, 1}
\]
on the set of critical points of $f$ by requiring
\[
\sign(w)=
\begin{cases}
1 & \text{ if $p_{*}|_{\xi_{Y_{w}}^{\perp}}$ is everywhere orientation preserving} \\
-1 &\text{ if $p_{*}|_{\xi_{Y_{w}}^{\perp}}$ is everywhere orientation reversing}.
\end{cases}
\]
Choosing at each $w\in\crit(f)$ with $\sign(w)=1$ an 
orientation preserving linear map $\Phi_{w}:T\Sigma_{w}\to (\R^{2}, \omega_{0})$
and 
at each $w\in\crit(f)$ with $\sign(w)=-1$ an 
orientation reversing linear map $\Phi_{w}:T\Sigma_{w}\to (\R^{2}, \omega_{0})$
we obtain a global orientation-preserving trivialization
\[
\Phi_{w}\circ p_{*}|_{\xi_{Y_{w}}^{\perp}}\to(\R^{2}, \omega_{0})
\]
of $\xi_{Y_{w}}^{\perp}$ which can be homotoped to a symplectic trivialization.
Thus in an $f$-admissible fibration there is a preferred homotopy class of
symplectic trivialization of the symplectic normal bundle
to $p^{-1}(\crit (f))$.

Assuming still that $M$ is a stable Hamiltonian manifold, 
let $\tilde J$ be an almost complex structure on $\R\times M$ which is compatible with the stable Hamiltonian
structure $\Ha=(\lambda, \omega)$; that is,
$\tl J$ is an $\R$-invariant endomorphism of $T(\R\times M)$ which squares to negative the identity, and
with respect to the splitting
\[
T(\R\times M)\approx \R\partial_{a}\oplus TM\approx \R\partial_{a}\oplus\R\X\oplus\xi
\]
the action of $\tl J$ is given by
\begin{equation}\label{e:R-invariant-extension}
\tl J\partial_{a}=\X \qquad\text{ and }\qquad \tl J|_{\pi^{*}\xi}=\pi^{*}J
\end{equation}
where $\pi:\R\times M\to M$ is the canonical projection and $J\in\End(\xi)$ is a complex structure
on $\xi$ for which the bilinear form $\omega(\cdot, J\cdot)|_{\xi\times\xi}$ is symmetric and positive definite.
We say that a codimension-$2$ foliation $\F$ of $\R\times M$ is an
\emph{$\R$-invariant,  asymptotically cylindrical, $\tilde J$-holomorphic foliation} if:
\begin{itemize}
\item $\mathcal{F}$ is invariant under translations in the $\R$-coordinate,
\end{itemize}
and if there exists a strong stable Hamiltonian hypersurface $V\subset M$
so that:
\begin{itemize}
\item $\R\times V$ has $\tilde J$-invariant tangent space,

\item the union of leaves of $\mathcal{F}$ fixed by $\R$-translation
is equal to $\R\times V\subset\R\times M$, and

\item all other leaves of the foliation are
$\tilde J$-holomorphic hypersurfaces which are asymptotically cylindrical\footnote{
\emph{Asymptotically cylindrical} will be defined precisely in Section \ref{s:intersections}.
}
over some collection of
components of $V$,
and which project by $\pi$ to embedded submanifolds smoothly foliating $M\setminus V$.
\end{itemize}
We will refer to the stable Hamiltonian hypersurface $V$ as the
\emph{binding set} of the foliation $\F$.
For brevity, we will from now on refer to $\R$-invariant, asymptotically cylindrical, $\tilde J$-holomorphic
foliations simply as \emph{holomorphic foliations} or \emph{$\tilde J$-holomorphic foliations} when we wish to
specify the almost complex structure.

Now assuming the manifold $M$ is equipped with both
a holomorphic foliation 
$\F$ with binding $V$ and an $f$-admissible fibration $(\Sigma, p, Y, f)$,
we will say that \emph{$\F$ is compatible with $(\Sigma, p, Y, f)$}
if:
\begin{itemize}
\item $p^{-1}(\crit(f))$ is equal to the binding $V$ of the foliation,

\item all other leaves of the foliation are diffeomorphic to $\R\times Y$ and admit smooth parametrizations of
the form
\[
(s, y)\in \R\times Y\mapsto (a(s, y), m(s, y))\in\R\times M
\]
where $p(m(s, y))=\gamma(s)$ for some solution $\gamma$ to the gradient flow equation 
\[
\dot\gamma(s)=\nabla f(\gamma(s))
\]
with respect to an appropriate metric $g_{\Sigma}$ on $\Sigma$, and where
$a(s, y)$ satisfies
\[
\lim_{s\to\pm\infty}a(s, y)=\sign(\lim_{s\to\pm\infty}p(m(s, y)))\pm\infty.
\]
\end{itemize}

We are interested in understanding the punctured pseudoholomorphic curves in 
$\R\times M$ when $M$ is equipped with a holomorphic foliation compatible with an $f$-admissible fibration
and all asymptotic limits of the given curve lie in the binding of the foliation.
We recall that a \emph{asymptotically cylindrical, punctured pseudoholomorphic map} 
is a quadruple $(S, j, \Gamma, \tl u=(a, u))$ where
\begin{itemize}
\item $(S, j)$ is a closed Riemann surface

\item $\Gamma\subset S$ is a finite set, called the set of \emph{punctures}

\item $\tl u=(a, u):S\setminus\Gamma\to\R\times M$ is a $\tl J$-holomorphic map, i.e.\ 
$\tl u$ satisfies the equation
\[
d\tl u\circ j=\tl J(\tl u)\circ d\tl u,
\]
and

\item For each puncture $z\in\Gamma$ there exists a periodic orbit $\gamma_{z}^{m_{z}}$
so that the map $\tl u$ is asymptotic near $z$ to a half-cylinder of the form $\R^{+}\times\gamma_{z}^{m_{z}}$ or
$\R^{-}\times\gamma_{z}^{m_{z}}$.  Here $\gamma_{z}$ is a simple periodic orbit of $\X$ and
$\gamma_{z}^{m_{z}}$ for $m_{z}\in\N$ denotes the $m_{z}$-fold cover of the simple orbit $\gamma_{z}$.
\end{itemize}
We will use the term \emph{pseudoholomorphic curve} to refer to
an equivalence class $C=[S, j, \Gamma, \tl u=(a, u)]$ of such maps
under the equivalence relation of holomorphic reparametrization of the domain.

We consider now an asymptotically cylindrical, punctured pseudoholomorphic map,
$(S, j, \Gamma, \tl u=(a, u))$ in a stable Hamiltonian manifold $(M, \Ha)$
equipped with an $f$-admissible fibration $(\Sigma, p, Y, f)$, and
we assume that all asymptotic limits $\gamma_{z}^{m_{z}}$ of the map $\tl u$ lie in
the strong stable Hamiltonian hypersurface
$p^{-1}(\crit(f))$.
In this case, the projection of the map $p\circ u:S\setminus\Gamma\to\Sigma$ admits
a continuous extension $\bar v:S\to\Sigma$ over the punctures.
Our main theorem puts restrictions on the degree of this map in the presence of a holomorphic foliation
and 
under some assumptions on the normal Conley--Zehnder indices of the asymptotic limits
and tells us that, under these assumptions, 
the image of the pseudoholomorphic map $\tl u$ is contained in a leaf of the foliation precisely when the
degree of this map is zero.
\begin{theorem}\label{t:main}
Let $(M, \lambda, \omega)$ be a nondegenerate stable Hamiltonian manifold
equipped with a compatible almost complex structure $\tilde J$,
an $f$-admissible fibration $(\Sigma, p, Y, f)$ and a $\tilde J$-holomorphic foliation
$\F$ compatible with $(\Sigma, p, Y, f)$.
Let $(S, j, \Gamma, \tilde u=(a, u))$ be an asymptotically cylindrical $\tl J$-holomorphic map
with all asymptotic limits contained in $p^{-1}(\crit(f))$ and let $\bar v:S\to\Sigma$ denote the continuous
extension of the map $p\circ u:S\setminus\Gamma\to\Sigma$ over the punctures.
Assume moreover that 
\begin{equation}\label{e:cz-condition}
\mu_N^{\Phi}(\gamma_{z}^{m_{z}})\in\br{-1, 0, 1}
\end{equation}
for every $z\in\Gamma$, where $\Phi$ is the homotopy class of symplectic trivialization of
the symplectic normal bundle to $p^{-1}(\crit(f))$ determined by the fibration.
Then:
\begin{itemize}
\item If $\sign(w)=1$ for all $w\in\crit(f)$ then the map $\bar v$ has nonnegative degree and
has degree zero
precisely when the image of the map $\tilde u$ is contained in a leaf of the foliation $\F$.

\item If $\sign(w)=-1$ for all $w\in\crit(f)$ then the map $\bar v$ has nonpositive degree and
has degree zero
precisely when the image of the map $\tilde u$ is contained in a leaf of the foliation $\F$.

\item If the map $\sign:\crit(f)\to\br{-1, 1}$ is surjective then the map $\bar v$ has degree zero and
the image of the map $\tilde u$ is contained in a leaf of the foliation $\F$.
\end{itemize}
\end{theorem}
The proof of our main result here relies heavily on some results from the in-preparation work
\cite{Siefring:highdim} which
generalizes some of the results from the $4$-dimensional intersection theory studied
in \cite{Siefring:2011} to higher dimensions.  We will summarize the relevant results in the following section
and then apply these results in Section \ref{s:main-proof} below to prove the main the result.

As an immediate corollary of the above result we have the following.

\begin{corollary}\label{c:main}
With $(M, \lambda, \omega)$, $\tilde J$, $(\Sigma, p, Y, f)$, $\F$, and $(S, j, \Gamma, \tilde u=(a, u))$
satisfying all the assumptions of Theorem \ref{t:main} above, assume that
$g(S)<g(\Sigma)$.
Then the image of the map $\tl u$ is contained in a leaf of the foliation $\F$.
\end{corollary}

\begin{proof}
As observed before the statement of Theorem \ref{t:main}, the assumption that all punctures of
$u:S\setminus\Gamma\to M$
are contained in the binding set of the foliation $V=p^{-1}(\crit(f))$ implies that the
map $p\circ u:S\setminus\Gamma\to\Sigma$ has a continuous extension
$\bar v:S\to\Sigma$.  A well-known argument
from algebraic topology
then implies that
the degree of the map $\bar v$ must be zero.
Indeed,
since for a closed, oriented surface the cup product pairing
\[
H^{1}(\Sigma)\times H^{1}(\Sigma)\stackrel{\cup}{\longrightarrow}H^{2}(\Sigma)
\]
is nondegenerate, a map
$\bar v:S\to\Sigma$ with nonzero degree induces an injection
\[
\bar v^{*}:H^{1}(\Sigma)\approx \Z^{2g(\Sigma)}\to H^{1}(S)\approx \Z^{2g(S)}
\]
which is impossible unless $g(\Sigma)\le g(S)$.
Since we assume that $g(S)< g(\Sigma)$ we conclude that $\bar v$ has degree zero.

Given that $\det(\bar v)=0$, it follows immediately from Theorem \ref{t:main} that
the image of $\tl u$ is contained in a leaf of the foliation.
\end{proof}

These results have applications in work of the first author on algebraic torsion in contact manifolds.
In \cite{Moreno:thesis, Moreno:torsion}, the author introduces a higher-dimensional generalization of the notion of a
\emph{spinal open book decomposition} (SOBD), as defined in \cite{LisiVHMWendl}
for dimension $3$.
This geometric structure supports a unique isotopy class of contact structures in the spirit of Giroux \cite{Giroux:2002}
and contains a fibration over a contact manifold with Liouville fibers (the ``pages'').
Given a SOBD supporting a contact structure, it induces holomorphic foliations in the symplectization of the contact
manifold lifting the pages of the SOBD, generalizing the construction of a \emph{holomorphic open book}
as e.g.\ in \cite{AbbasCieliebakHofer:2005,Wendl:2010ob}. In the case where the leaves of the foliation are codimension-2
--- which is the case where the intersection-theoretic arguments used here are applicable
--- it is of the type described above. 

In order to develop computational techniques for SFT-type invariants of the contact manifold, in which the control over 
holomorphic curves is crucial, the above results are important. In the situation described above,
one can arrange that the relevant holomorphic curves are asymptotic to periodic orbits
in fibers
lying over critical points of a Morse function.
Moreover, the normal linearized flow along these fibers
can be expressed in terms of the Hessian of the Morse function $f$. In particular, 
one can show that given a number $T>0$, every periodic orbit
$\gamma\in Y_{w}$ with period less than $T$ has normal Conley--Zehnder index given by the formula
\[
\mu_N^{\Phi}(\gamma)=\sign(w)\bp{\ind_{w}(f)-1}\in\br{-1, 0, 1}
\]
provided the function $f$ is sufficiently $C^{2}$ close to a constant,
so the hypotheses of our main theorem are met.
In the case where the degree of the resulting map $\bar v$ is zero ---
as it is, for example, in the case considered in Corollary \ref{c:main} above ---
then the above theorem reduces the study of certain holomorphic curves in the ambient
symplectization to that of the Liouville completion of the pages,
thus reducing the problem by two dimensions.
This fact, when combined with the symmetries of the Morse function $f$,
can be exploited to
obtain information on the contact structure as is done in  \cite{Moreno:thesis, Moreno:torsion}.

\begin{ack}
We'd like to thank Murat Sa\u{g}lam for some helpful feedback on an earlier draft of this work. 
\end{ack}

\section{Intersection theory of punctured pseudoholomorphic curves and pseudoholomorphic hypersurfaces}\label{s:intersections}
In this section we will review some results from the in-preparation work \cite{Siefring:highdim} which
are needed in the proof of our main result.

Let $(M^{2n+1}, \Ha=(\lambda, \omega))$ be a closed, orientable manifold equiped with a 
nondegenerate stable Hamiltonian structure.
Recall from that the introduction that a submanifold
$i:V^{2n-1}\hookrightarrow M$ is said to be a stong stable Hamiltonian hypersurface
if $\Ha'=(\lambda', \omega'):=(i^{*}\lambda, i^{*}\omega)$ is a stable Hamiltonian structure on
$V$ and the Reeb vector field $\X$ of $\Ha$ is everywhere tangent to $V$.
In this case
we have a splitting
\begin{equation}\label{e:splitting}
TM|_{V}=\R \X\oplus(\xi', \omega')\oplus (\xi_{V}^{\perp}, \omega)
\approx\R X_{\Ha'}\oplus (\xi\cap TV, \omega)\oplus (\xi_{V}^{\perp}, \omega)
\end{equation}
with $\xi_{V}^{\perp}$ the symplectic complement to $\xi'$ in $\xi|_{V}$ as defined 
in \eqref{e:symp-comp-def}, and we note that the first two summands give $TV$.
The linearized flow of $\X$ along $V$ preserves this splitting along with
the symplectic structure on the
second two summands.

Let $\gamma:S^{1}\approx\R/\Z\to M$ be a $T$-periodic orbit of $\X$, i.e.\ $\gamma$ satisfies the equation
\[
\dot\gamma(t)=T\cdot \X(\gamma(t))
\]
for all $t\in S^{1}$.
Assuming that 
$\gamma(S^{1})\subset V$, 
we can choose a symplectic trivialization of the
hyperplane distribution
$\xi=\xi'\oplus \xi_{V}^{\perp}|_{\gamma(S^{1})}$ along $\gamma$ which respects the splitting
\eqref{e:splitting}, i.e.\ one of the form
\[
\Phi=\Phi_{T}\oplus\Phi_{N}:\xi'\oplus \xi_{V}^{\perp}|_{\gamma(S^{1})} \to
S^{1}\times (\R^{2n-2}, \omega_{0})\oplus(\R^{2}, \omega_{0})
\]
with $\Phi_{T}$ and $\Phi_{N}$ symplectic trivialization of $\xi'|_{\gamma(S^{1})}$ and
$\xi_{V}^{\perp}|_{\gamma(S^{1})}$ respectively.
Given such a trivialization we can define the Conley--Zehnder index of the orbit
$\gamma$ viewed as an orbit in $M$ as usual by
\[
\mu^{\Phi}(\gamma)
:=\mu_{CZ}(\Phi(\gamma(t))\circ d\psi_{Tt}(\gamma(0))\circ\Phi(\gamma(0))^{-1})
\]
where $\psi:\R\times M\to M$ is the flow generated by $\X$ and where
$\mu_{CZ}$ on a path of symplectic matrices starting at the identity and ending at a matrix without $1$
in the spectrum is as defined in \cite[Theorem 3.1]{HWZ:prop2}.  But since, as observed above,
$d\psi_{t}$ preserves the splitting \eqref{e:splitting}, we can also consider the
Conley--Zenhder indices that arise from the restrictions of $d\psi_{t}$ to $\xi'|_{\gamma}$ and
$\xi_{V}^{\perp}|_{\gamma}$.
In particular we define
\[
\mu^{\Phi_{T}}_{V}(\gamma)
:=\mu_{CZ}(\Phi_{T}(\gamma(t))\circ d\psi_{Tt}(\gamma(0))|_{\xi'}\circ\Phi_{T}(\gamma(0))^{-1})
\]
which is the Conley--Zehnder index of $\gamma$ viewed as a periodic orbit lying in $V$, and
\[
\mu_{N}^{\Phi_{N}}(\gamma)
:=\mu_{CZ}(\Phi_{N}(\gamma(t))\circ d\psi_{Tt}(\gamma(0))|_{\xi_{V}^{\perp}}\circ\Phi_{N}(\gamma(0))^{-1})
\]
which we will call the \emph{normal Conley--Zehnder index of $\gamma$ relative to $\Phi_{N}$}.
We note that basic properties of Conley--Zehnder indices which can be found, e.g.\ in
\cite[Theorem 3.1]{HWZ:prop2} show that these quantites are related by
\[
\mu^{\Phi}(\gamma)=\mu_{V}^{\Phi_{T}}(\gamma)+\mu_{N}^{\Phi_{N}}(\gamma).
\]

We now recall that a complex structure $J$ on $\xi$ is said to be \emph{compatible} with the
stable Hamiltonian structure $(\lambda, \omega)$
if the bilinear form on $\xi$ defined by
$\omega(\cdot, J\cdot)|_{\xi\times\xi}$
is symmetric and positive definite.
We will denote the set of compatible complex structures on $\xi$ by $\J(M, \xi)$.
Given a
strong stable Hamiltonian hypersurface
$V\subset M$
and a choice of $J\in\J(M, \xi)$ we will say that $J$ is \emph{$V$-compatible} if $J$ fixes the
hyperplane distribution
$\xi'=TV\cap\xi$ along $V$.
We will denote the set of such complex structures by
$\J(M, V, \xi)$.
We claim that a $J\in\J(M, V, \xi)$ necessarily also fixes
$\xi_{V}^{\perp}$.  Indeed, compatibility of $J$ with
$\omega$
implies that
$\omega(J\cdot, J\cdot)|_{\xi\times\xi}=\omega|_{\xi\times \xi}$
since we can compute
\begin{align*}
\omega(Jv, Jw)&
=\omega(w, J(Jv))
& \text{symmetry of $\omega(\cdot, J\cdot)|_{\xi\times\xi}$} \\
&=\omega(w, -v)
& J^{2}=-I \\
&=\omega(v, w)
\end{align*}
for any sections $v$, $w$ of $\xi$.
Thus if $v\in\xi|_{V}$ is
$\omega$-orthogonal
to every vector in $\xi'=\xi\cap TV$, then $Jv$ is also
$\omega$-orthogonal
to every vector in $\xi'$ provided that $J$ fixes $\xi'$.
Thus $J$ fixes $\xi_{V}^{\perp}$ as claimed.
We note that since both the linearized flow $d\psi_{t}$ of $\X$
and a compatible $J\in\J(M, V, \xi)$ preserve the splitting
\eqref{e:splitting}, the asymptotic operator
\[
\A_\gamma h(t):=-J\left.\frac{d}{ds}\right|_{s=0}d\psi_{-Ts}h(t+s)
\]
of a periodic orbit $\gamma$ lying in $V$ also preserves the splitting.  We will write
\[
\A_\gamma=\A_\gamma^{T}\oplus\A_\gamma^{N}:W^{1,2}(\xi')\oplus W^{1,2}(\xi_{V}^{\perp})
\to L^{2}(\xi')\oplus L^{2}(\xi_{V}^{\perp})
\]
to indicate the resulting splitting of the operator.

Continuing to assume that $V\subset M$ is a strong
stable Hamiltonian hypersurface and
$J\in\J(M, V, \xi)$ is a $V$-compatible complex structure, we extend 
$J$ to an $\R$-invariant almost complex structure $\tilde J$ on $\R\times M$
in the usual way, i.e.\ so that $\tilde J$ satisfies \eqref{e:R-invariant-extension}.
We note that for such an almost complex structure, the submanifold
$\R\times V$ of $\R\times M$ is $\tilde J$-holomorphic since
we assume that $\X$ is tangent to $V$ and that $J$ fixes
$\xi'=\xi\cap TV$.
Just as one can consider holomorphic curves which are asymptotic to cylinders of the 
form $\R\times\br{\text{periodic orbit}}$ one can $\tilde J$-holomorphic hypersurfaces which
are asymptotic to cylindrindical $\tilde J$-holomorphic hypersurfaces of the form
$\R\times V$ with $V$ a
strong stable Hamiltonian hypersurface.
Before giving a more precise definition,
we introduce some more geometric data on our manifold.

Given a $J\in\J(M, \xi)$ we can define a Riemannian metric
\begin{equation}\label{e:metric-M}
g_{J}(v, w)=\lambda(v)\lambda(w)+\omega(\pi_{\xi}v, J\pi_{\xi}w)
\end{equation}
where
$\pi_{\xi}:TM\approx\R \X\oplus\xi\to\xi$
is the projection onto $\xi$ along $\X$.
We can extend $g_{J}$ to a metric $\tilde g_{J}$ on $\R\times M$
by forming the product metric with the standard metric on $\R$,
i.e.\ by defining
\begin{equation}\label{e:metric-RxM}
\tilde g_{J}:=da\otimes da+\pi^{*}g_{J}.
\end{equation}
We will denote the exponential maps of
$g_{J}$ and $\tilde g_{J}$ by $\exp$ and
$\widetilde\exp$ respectively and note that these are related by
\[
\widetilde\exp_{(a, p)}(b, v)=(a+b, \exp_{p}v).
\]
We note that if $J$ is $V$-compatible for some strong
stable Hamiltonian hypersurface $V\subset M$,
then the symplectic normal bundle $\xi_{V}^{\perp}$ is the
$g_{J}$-orthogonal complement of $TV$ in $TM|_{V}$, and that
$\pi^{*}\xi_{V}^{\perp}$ is the $\tilde g_{J}$-orthogonal complement of $T(\R\times V)$ in
$T(\R\times M)|_{\R\times V}$.
We further note that, since $V$ is assumed to be compact, the restrictions of
$\exp$ and $\widetilde\exp$ to $\xi_{V}^{\perp}$ and $\pi^{*}\xi_{V}^{\perp}$ respectively
are embeddings on some neighborhood of the zero sections.

Now consider a pair $V_{+}$, $V_{-}$ of
strong stable Hamiltonian hypersurfaces
and assume that
$V:=V_{+}\cup V_{-}$ is also a
strong stable Hamiltonian hypersurface,
i.e.\ that all components of $V_{+}$ and $V_{-}$ are either disjoint or identical.
We let $J\in\J(M, V, \xi)$ be a $V$-compatible $J$ with associated $\R$-invariant almost complex structure
$\tilde J$ on $\R\times M$.
We are interested in $\tilde J$-holomorphic submanifolds which outside of a compact set can be described
by exponentially decaying sections of the normal bundles to $V_{+}$ and $V_{-}$,  More precisely,
we say that a $\tilde J$-holomorphic submanifold $\tilde V\subset\R\times M$ is
\emph{positively asymptotically cylindrical over $V_{+}$} and
\emph{negatively asymptotically cylindrical over $V_{-}$} if there exists an $R>0$ and sections
\begin{gather*}
\eta_{+}:[R, +\infty)\to C^\infty(\xi_{V_{+}}^{\perp}) \\
\eta_{-}:(-\infty, -R] \to C^\infty(\xi_{V_{-}}^{\perp}) \\
\end{gather*}
so that
\begin{gather*}
\tilde V\cap\left([R, +\infty)\times M\right)=\bigcup_{(a, p)\in [R, +\infty)\times V_{+}} \widetilde\exp_{(a, p)}\eta_{+}(a, p) \\
\tilde V\cap\left((-\infty, -R]\times M\right)=\bigcup_{(a, p)\in (-\infty, -R]\times V_{-}} \widetilde\exp_{(a, p)}\eta_{-}(a, p)
\end{gather*}
and so that there exist constants $M_{i}>0$, $d>0$ satisfying
\[
\abs{\widetilde\nabla^{i}\eta_{\pm}(a, p)}\le M_{i}e^{-d \abs{a}}
\]
for all $i\in\N$ and $\pm a\in [R, +\infty)$, where
$\tilde\nabla$ is the extension of a connection $\nabla$ on  $\xi_{V}^{\perp}$
to a connection $\tilde \nabla$ on $\pi^{*}\xi_{V}^{\perp}$ defined by requiring
$\tilde\nabla_{\partial_{a}}\eta(a, p)=\partial_{a}\eta(a, p)$.
We will refer to the sections $\eta_{+}$ and $\eta_{-}$ respectively as \emph{positive}
and \emph{negative asymptotic representatives of $\tilde V.$}

Our main goal here is to understand the intersection properties of punctured
pseudoholomorphic curves with
asymptotically cylindrical pseudoholomorphic hypersurfaces.  The main difficulty 
arises from the noncompactness of the manifolds in question.
Indeed a punctured pseudoholomorphic curve whose image is not contained in the
$\tilde J$-holomorphic hypersurface $\tilde V$ may have punctures limiting to a periodic orbits lying in
$V_{+}$ or $V_{-}$.
In this case, it's not a priori clear that the intersection number between the curve and the hypersurface is finite.
Even assuming this intersection number is finite, it is not homotopy invariant as intersections can be lost or created
at infinity.
We will see below that these difficulties can be dealt with via higher-dimensional analogs of
techniques developed in \cite{Siefring:2011}.

Before presenting the relevant results it will be convenient
to establish some standard assumptions and notations for the next several
definitions and results.

\begin{assumptions}\label{standing-assumptions}
We assume that:
\begin{enumerate}\renewcommand{\theenumi}{\alph{enumi}}
\item $(M, \lambda, \omega)$ is a closed, manifold with nondegenerate stable Hamiltonian structure
$(\lambda, \omega)$ with $\xi=\ker\lambda$ and $\X$ the associated Reeb vector field,

\item $V_{+}\subset M$, $V_{-}\subset M$, and $V=V_{+}\cup V_{-}$ are
strong stable Hamiltonian hypersurfaces
of $M$,

\item $J\in\J(M, V, \xi)$ is a $V$-compatible complex structure on $\xi$ and
$\tilde J$ is the $\R$-invariant almost complex structure on $\R\times M$ associated to $J$ (defined by
\eqref{e:R-invariant-extension},

\item $\tilde g_{J}$ is the Riemannian metric on $\R\times M$ defined by \eqref{e:metric-RxM} and
$\widetilde\exp$ is the associated exponential map,

\item $\tilde V\subset\R\times M$ is a $\tilde J$-holomorphic hypersurface which is positively asymptotically cylindrical over $V_{+}$ and negatively asymptotically cylindrical over $V_{-}$,

\item $\eta_{+}:[R, +\infty)\to C^\infty (\xi_{V_{+}}^{\perp})$ and
$\eta_{-}:(-\infty, -R]\to C^{\infty}(\xi_{V_{-}}^{\perp})$ are, respectively, positive and negative asymptotic representatives of $\tilde V$,

\item $C=[S, j, \Gamma=\Gamma^{+}\cup\Gamma^{-}, \widetilde{u}=(a, u)]$ is a finite-energy $\tilde J$-holomorphic curve, and
at $z\in\Gamma$, $C$ is asymptotic to $\gamma_{z}^{m_{z}}$
(with $\gamma_{z}^{m_{z}}$ indicating the $m_{z}$-fold covering of a simple periodic orbit $\gamma_{z}$),

\item $\Phi$ is a trivialization of $\xi_{V}^{\perp}$ along every periodic orbit lying in $V$
which occurs as an asymptotic limit of $C$. 

\end{enumerate}
\end{assumptions}

The following theorem can be seen as a generalization of
Theorem 2.2 in \cite{Siefring:2008}.  Proof will be given in \cite{Siefring:highdim}.

\begin{theorem}\label{t:normal-asymptotics}
Assume \ref{standing-assumptions} and 
assume that at $z\in\Gamma^{+}$, $C$ is asymptotic to a periodic orbit
$\gamma_{z}^{m_{z}}\subset V_{+}$.
Then there exists an $R'\in\R$, a smooth map
\[
u_{T}:[R', \infty)\times S^{1}\to [R, \infty)\times V_{+}
\]
and a smooth section
\[
u_{N}:[R', \infty)\times S^{1}\to u_{T}^{*}\pi^{*}\xi_{V_{+}}^{\perp}
\]
so that the map
\begin{equation}
(s, t)\mapsto\widetilde\exp_{u_{T}(s, t)}u_{N}(s, t)
\end{equation}
parametrizes $C$ near $z$.
Moreover,
if we assume that the image of $C$ is not a subset of the asymptotically cylindrical hypersurface $\tilde V$,
then
\begin{equation}\label{e:normal-asymptotics}
u_{N}(s, t)-\eta_{+}(u_{T}(s, t))=e^{\mu s}[e(t)+r(s, t)]
\end{equation}
for all $(s, t)\in[R', \infty)\times S^{1}$
where:
\begin{itemize}
\item $\mu<0$ is a negative eigenvalue of the normal asymptotic operator
$\A_{\gamma_{z}^{m_{z}}}^{N}$,
\item $e\in\ker(\mathbf{A}_{\gamma_{z}^{m_{z}}}^{N}-\mu)\setminus\br{0}$ is an eigenvector with eigenvalue $\mu$, and
\item $r:[R', \infty)\times S^1 \to u_{T}^{*}\pi^{*}\xi_{V}^{\perp}$ is a smooth section satisfying
exponential decay estimates of the form
\begin{equation}\label{e:normal-asymptotics-remainder}
\abs{\tilde\nabla^{i}_{s}\tilde\nabla^{j}_{t} r(s, t)}\le M_{ij}e^{-d\abs{s}}
\end{equation}
for some positive contants $M_{ij}$, $d$ and all $(i, j)\in\N^{2}$
\end{itemize}

Similarly, if we 
assume that at $z\in\Gamma^{-}$, $C$ is asymptotic to a periodic orbit
$\gamma_{z}^{m_{z}}\subset V_{-}$., then there exists an $R'\in\R$, a smooth map
\[
u_{T}:(-\infty, R'] \times S^{1}\to (-\infty, -R]\times V_{-}
\]
and a smooth section
\[
u_{N}:(-\infty, R'] \times S^{1}\to u_{T}^{*}\pi^{*}\xi_{V_{-}}^{\perp}
\]
so that the map
\[
(s, t)\mapsto\widetilde\exp_{u_{T}(s, t)}u_{N}(s, t)
\]
parametrizes $C$ near $z$.
Moreover, if the image of $C$ is not contained in $\tilde V$, then 
$u_{N}(s, t)-\eta_{-}(u_{T}(s, t))$ satisfies a formula of the form \eqref{e:normal-asymptotics}
for all $(s, t)\in(-\infty, R']\times S^{1}$, where now:
\begin{itemize}
\item $\mu>0$ is a positive eigenvalue of the normal asymptotic operator
$\A_{\gamma_{z}^{m_{z}}}^{N}$,
\item $e\in\ker(\A_{\gamma_{z}^{m_{z}}}^{N}-\mu)\setminus\br{0}$, as before,
is an eigenvector with eigenvalue $\mu$, and
\item $r:(-\infty, R']\to u_{T}^{*}\pi^{*}\xi_{V}^{\perp}$ is a smooth section satisfying
exponential decay estimates of the form \eqref{e:normal-asymptotics-remainder}
for some positive contants $M_{ij}$, $d$.
\end{itemize}
\end{theorem}

We note that the bundles of the form
$u_{T}^{*}\pi^{*}\xi_{V}^{\perp}$
ocurring in the statement of this theorem are trivializable
since they are complex line bundles over a space which retracts onto $S^{1}$.
In any trivialization the eigenvector $e$ from formula \eqref{e:normal-asymptotics}
satisfies a linear, nonsingular ODE,
and thus is nowhere vanishing since we assume it is not identically zero.
Since the ``remainder term'' $r$ in the formula \eqref{e:normal-asymptotics}
converges to zero, we conclude that the functions
$u_{N}(s, t)-\eta_{\pm}(u_{T}(s, t))$ are nonvanishing for sufficiently large $\abs{s}$.
However, since zeroes of this function can be seen to correspond to intersections between
the curve $C$ and the hypersurface $\tilde V$ occuring sufficiently close to the punctures of $C$, we conclude
that all intersections between $C$ and $\tilde V$ are contained in a compact set.
Moreover, since intersections between $C$ and $\tilde V$ can be shown to be isolated and of positive local order
(see e.g.\  \cite[Lemma 3.4]{IonelParker:2003}),
we conclude that the algebraic intersection number between $C$ and $\tilde V$ is finite:

\begin{corollary}\label{c:finite-intersections}
Assume \ref{standing-assumptions} and assume that
no component of  the curve $C$
has image contained in 
the $\tilde J$-holomorphic hypersurface $\tilde V$.  Then the algebraic intersection number
$C\cdot \tilde V$, defined by summing local intersection indices, is finite and nonnegative, and
$C\cdot \tilde V=0$ precisely when $C$ and $\tilde V$ do not intersect.
\end{corollary}

This corollary deals with the first difficulty in understanding intersections between punctured curves and asymptotically cylindrical hypersurfaces described above, namely the finiteness of the intersection number. A second consequence of the asymptotic formula from Theorem \ref{t:normal-asymptotics}, again stemming
from the fact that the quantities $u_{N}(s, t)-\eta_{\pm}(u_{T}(s, t))$ are nonzero for sufficiently large
$\abs{s}$, is that the normal approach of the curve $C$ has a well-defined winding number
relative to a trivialization $\Phi$ of $\xi_{V}^{\perp}|_{\gamma_{z}}$.  This winding will be given by
the winding of the eigenvector from formula \eqref{e:normal-asymptotics} relative to $\Phi$, and for a given 
puncture $z$ of $C$, we will denote this quantity by
\[
\wind_{rel}^\Phi((C; z), \tilde V)=\wind(e).
\]
Combining this observation with the characterization of the Conley--Zehnder
index in terms of the asymptotic operator from \cite[Definition 3.9/Theorem 3.10]{HWZ:prop2} leads
to the following corollary.

\begin{corollary}\label{c:normal-rel-winding}
Assume \ref{standing-assumptions}, and assume that
no component of the curve
$C=[S, j, \Gamma_{+}\cup\Gamma_{-}, \tilde u=(a, u)]$
has image contained in
the holomorphic hypersurface $\tilde V$.  Then:
\begin{itemize}
\item If $z\in\Gamma$ is a positive puncture at which $\tilde u$ limits to $\gamma_z^{m_z}\subset V_{+}$ then
\begin{equation}\label{e:wind-rel-pos}
\wind_{rel}^{\Phi}((C; z), \tilde V)\le \fl{\mu_{N}^{\Phi}(\gamma_z^{m_z})/2}=:\alpha_N^{\Phi;-}(\gamma_z^{m_z}).
\end{equation}

\item If $z\in\Gamma$ is a negative puncture at which $\tilde u$ limits to $\gamma_z^{m_z}\subset V_{-}$ then
\begin{equation}\label{e:wind-rel-neg}
\wind_{rel}^{\Phi}((C; z), \tilde V)\ge \ceil{\mu_{N}^{\Phi}(\gamma_z^{m_z})/2}=:\alpha_N^{\Phi;+}(\gamma_z^{m_z}).
\end{equation}
\end{itemize}
\end{corollary}

The numbers $\alpha_N^{\Phi;-}(\gamma)$ and $\alpha_N^{\Phi;+}(\gamma)$ are, respectively, the biggest/smallest winding number achieved by an eigenfunction of the normal asymptotic operator of any orbit $\gamma$ corollary responding to a negative/positive eigenvalue. Observe that we have the formulas 
\begin{equation}\label{alphas}
\mu_{N}^{\Phi}(\gamma)= 2\alpha_N^{\Phi;-}(\gamma)+p_N(\gamma)=2\alpha_N^{\Phi;+}(\gamma)-p_N(\gamma),
\end{equation}
where $p_N(\gamma) \in \{0,1\}$ is the \emph{normal parity} of the orbit $\gamma$ (which is independent of the trivialization $\Phi$).

We will see in a moment that this corollary
can be used to 
deal with the second difficulty in understanding intersections between 
punctured curves and asymptotically cylindrical hypersurfaces described above, namely, the fact
that the algebraic intersection number may not be invariant under homotopies.
We first introduce some terminology.
Assuming again \ref{standing-assumptions} and that
no component of
$C$ is a subset of $\tilde V$, we define the asymptotic
intersection number at the punctures of $C$ in the following way:
\begin{itemize}
\item
If for the positive puncture $z\in\Gamma_{+}$, $\gamma^{m_z}_{z}\subset V_{+}$, we define the 
\emph{asymptotic intersection number $\delta_{\infty}((C; z); \tilde V)$ of $C$ at $z$ with $\tilde V$} by
\begin{equation}\label{e:local-asymp-inum-pos}
\delta_{\infty}((C; z), \tilde V)=\fl{\mu_{N}^{\Phi}(\gamma_z^{m_z})/2}-\wind_{rel}^{\Phi}((C; z), \tilde V).
\end{equation}

\item 
If for the negative puncture $z\in\Gamma_{-}$, $\gamma^{m_z}_{z}\subset V_{-}$, we define the 
\emph{asymptotic intersection number $\delta_{\infty}((C; z); \tilde V)$ of $C$ at $z$ with $\tilde V$} by
\begin{equation}\label{e:local-asymp-inum-neg}
\delta_{\infty}((C; z), \tilde V)=\wind_{rel}^{\Phi}((C; z), \tilde V)-\ceil{\mu_{N}^{\Phi}(\gamma^{m_z}_{z})/2}.
\end{equation}

\item For all other punctures $z\in\Gamma_{\pm}$
(i.e. those for which $\gamma_{z}$ is not contained in $V_{\pm}$),
we define
\begin{equation}\label{e:local-asymp-inum-zero}
\delta_\infty((C; z), \tilde V)=0.
\end{equation}
\end{itemize}
We then define the \emph{total asymptotic intersection number of $C$ with $\tilde V$} by
\begin{equation}\label{e:global-asymp-inum}
\delta_{\infty}(C, \tilde V)=\sum_{z\in\Gamma}\delta_{\infty}((C; z), \tilde V).
\end{equation}

We observe that as a result of Corollary \ref{c:normal-rel-winding} the local and total asymptotic intersection numbers are always nonnegative.

\vspace{0.5cm}

Continuing to assume \ref{standing-assumptions} (but no longer necessarily that no component of
the curve $C$ has image contained in $\tilde V$),
we can use the trivialization $\Phi$ of $\xi_{V}^{\perp}$ along the asymptotic periodic orbits of $C$ lying
in $V$ to construct a perturbation $C_{\Phi}$ of $C$ in the following way.
For each puncture $z\in\Gamma$ for which the asymptotic limit $\gamma_{z}^{m_{z}}$ lies in $V$,
we first extend $\Phi$ to a trivialization
$\Phi:\xi_{V}^{\perp}|_{U_{z}}\to U_{z}\times\R^{2}$ on some open neighborhood 
$U_{z}\subset V$ of the asymptotic limit $\gamma_{z}$.  Then we consider the asymptotic parametrization
\[
(s, t)\mapsto \widetilde\exp_{u_{T}(s, t)}u_{N}(s, t)
\]
from Theorem \ref{t:normal-asymptotics} above
for $(s, t)\in [R,+\infty)\times S^{1}$ or $(-\infty, -R]\times S^{1}$ as appropriate,
where $R>0$ is chosen large enough so that $u_{T}$ has image contained in the neighborhood $U_{z}$ of
$\gamma_{z}$ on which the trivialization $\Phi$ has been extended.
We then perturb the map by replacing the above parametrization of $C$ near $z$ by the map
\[
(s, t)\mapsto \widetilde\exp_{u_{T}(s, t)}\bp{u_{N}(s, t)+\beta(\abs{s})\Phi(u_{T}(s, t))^{-1}\varepsilon}
\]
where $\beta:[0, \infty)\to[0, 1]$ is a smooth cut-off function equal to $0$ for $s<\abs{R}+1$ and equal to
$1$ for $\abs{s}>\abs{R}+2$, and $\varepsilon \ne 0$ is thought of as a number in $\C\approx\R^{2}$.
Given this, we can then define the 
\emph{relative intersection number $i^{\Phi}(C, \tilde V)$
of $C$ and $\tilde V$ relative to the the trivialization $\Phi$}
by
\[
i^{\Phi}(C, \tilde V):=C_{\Phi}\cdot \tilde V.
\]
It can be shown that this number is independent of choices made in the construction of $C_{\Phi}$ provided
the perturbations are sufficiently small.

Again continuing to assume \ref{standing-assumptions},
we now define the
\emph{holomorphic intersection product of $C$ and $\tilde V$} by
\[
C*\tilde V:=i^{\Phi}(C,\tilde V)
+\sum_{\stackrel{z\in\Gamma_{+}}{\gamma_{z}\subset V_{+}}}\fl{\mu_{N}^{\Phi}(\gamma_{z}^{m_{z}})/2}
-\sum_{\stackrel{z\in\Gamma_{-}}{\gamma_{z}\subset V_{-}}}\ceil{\mu_{N}^{\Phi}(\gamma_{z}^{m_{z}})/2}
\]
The key facts about the holomorphic intersection product are now given in the following theorem which generalizes
\cite[Theorem 2.2/4.4]{Siefring:2011}.
\begin{theorem}[Generalized positivity of intersections]\label{t:intersection-positivity}
With $\tilde V$ and $C$ as in \ref{standing-assumptions}, assume that $C$ is not contained in $\tilde V$,
the holomorphic intersection product $C*\tilde V$ depends only on the relative homotopy classes of $C$ and $\tilde V$.
Moreover, if the image of $C$ is not contained in $\tilde V$, then
\[
C*\tilde V=C\cdot \tilde V+\delta_{\infty}(C, \tilde V)\ge 0
\]
where $C\cdot \tilde V$ is the algebraic intersection number, defined by summing local intersection indices,
and $\delta_{\infty}(C, \tilde V)$ is the total asymptotic intersection number, defined by
\eqref{e:local-asymp-inum-pos}--\eqref{e:global-asymp-inum}.
In particular, $C*\tilde V\ge 0$ and equals zero if and only if $C$ and $\tilde V$ don't intersect and
all asymptotic intersection numbers are zero.
\end{theorem}
The proof of this theorem follows along very similar lines to Theorem 2.2/4.4 in \cite{Siefring:2011}.
The essential point is that the relative intersection number can be shown to be given by the formula
\[
i^{\Phi}(C, \tilde V)=
C\cdot \tilde V
-\sum_{\stackrel{z\in\Gamma_{+}}{\gamma_{z}\subset V_{+}}}\wind^\Phi_{rel}((C; z), \tilde V)
+\sum_{\stackrel{z\in\Gamma_{-}}{\gamma_{z}\subset V_{-}}}\wind^\Phi_{rel}((C; z), \tilde V).
\]
The result will then follow from Corollary \ref{c:normal-rel-winding} above. Detailed proof will be given in \cite{Siefring:highdim}.

Analogous to the case in four dimensions studied in \cite{Siefring:2011}, the
$\R$-invariance of the almost complex structure in the set-up here allows one to compute
the holomorphic intersection number and (in some cases) the algebraic intersection
number of a holomorphic curve and a holomorphic hypersurface with respect
to asymptotic winding numbers and intersections of each object with the asymptotic limits
of the other.

Before stating the relevant results we will first make some additional assumptions. We will henceforth assume that:
\begin{assumptions}\label{assumptions-2}$\;$
\begin{enumerate}\renewcommand{\theenumi}{\alph{enumi}}
\item $V_{+}$ and $V_{-}$ are disjoint,
\item $\xi_{V}^{\perp}$ of $V=V_{+}\cup V_{-}$ is trivializable, 
\item $\Phi:\xi_{V}^{\perp}\to V\times \R^{2}$ is a global trivialization,
\item $\tilde V$ is connected, and
\item the projection $\pi(\tilde V)$ of $\tilde V$ to $M$ is an embedded codimension-$1$ submanifold of
$M\setminus V$.
\end{enumerate}
\end{assumptions}
Under these assumptions, $\tilde V$ has a well-defined winding
$\wind_{\infty}^{\Phi}(\tilde V, \gamma)$ relative to $\Phi$
around any orbit $\gamma\subset V=V_{+}\cup V_{-}$
which can be defined by considering the asymptotic representatives $\eta_{+}$ or $\eta_{-}$ as appropriate
and computing
\begin{equation}\label{e:V-gamma-wind-1}
\wind_{\infty}^{\Phi}(\tilde V, \gamma)=\lim_{\abs{s}\to\infty}\wind \Phi^{-1}\eta_{\pm}(s, \gamma(\cdot)),
\end{equation}
or, equivalently by
\begin{equation}\label{e:V-gamma-wind-2}
\wind_{\infty}^{\Phi}(\tilde V, \gamma)=\wind_{rel}^{\Phi}((\R\times\gamma; \pm\infty), \tilde V).
\end{equation}
As in Corollary \ref{c:normal-rel-winding} above, it follows from the
asymptotic formula from Theorem \ref{t:normal-asymptotics} above and the characterization of the
Conley--Zehnder index from \cite{HWZ:prop2} that
\[
\wind_{\infty}^{\Phi}(\tilde V, \gamma)\le \fl{\mu_{N}^{\Phi}(\gamma)/2}=\alpha_N^{\Phi;-}(\gamma)
\]
if $\gamma\subset V_{+}$ and
\[
\wind_{\infty}^{\Phi}(\tilde V, \gamma)\ge \ceil{\mu_{N}^{\Phi}(\gamma)/2}=\alpha_N^{\Phi;+}(\gamma)
\]
if $\gamma\subset V_{-}$.

The following theorem, which can be seen as the higher dimensional version of
\cite[Corollary 5.11]{Siefring:2011}, gives a computation of the algebraic intersection number of
$\tilde V$ and $\R$-shifts of the curve $C$ in terms of the asymptotic data, and the intersections of
each object with the asymptotic limits of the other.
Proof will be given in \cite{Siefring:highdim}

\begin{theorem}\label{t:intersection-computation}
Assume \ref{standing-assumptions} and \ref{assumptions-2} and that the curve $C$ is connected and
not equal to an orbit cylinder and not contained in $\R\times V$.
For $c\in\R$, denote by $C_{c}$ the curve obtained from translating $C$ in the $\R$-coordinate by $c$.
Then for all but a finite number of value of $c\in\R$, the algebraic intersection number
$C_{c}\cdot \tilde V$ is given by the formulas:

\begin{equation*}
\begin{split}
C_{c}\cdot \tilde V
&=C\cdot (\R\times V_{+}) \\
&\hskip.25in
+\sum_{\stackrel{z\in\Gamma_{+}}{\gamma_{z}\in V_{+}}} 
\bp{\max\br{m_{z}\wind_{\infty}^{\Phi}(\tilde V, \gamma_{z}), \wind^{\Phi}_{rel}((C;z), \R\times V_{+})}-\wind^{\Phi}_{rel}((C;z), \R\times V_{+})} \\
&\hskip.25in
+ \sum_{z\in\Gamma_{-}}m_{z}(\R\times \gamma_{z})\cdot \tilde V  \\
&\hskip.25in
+\sum_{\stackrel{z\in\Gamma_{-}}{\gamma_{z}\in V_{-}}}
\bp{m_{z}\wind_{\infty}^{\Phi}(\tilde V, \gamma_{z})
-
\min\br{m_{z}\wind_{\infty}^{\Phi}(\tilde V, \gamma_{z}), \wind^{\Phi}_{rel}((C;z), \R\times V_{-})}} \\
&\hskip.25in
+\sum_{\stackrel{z\in\Gamma_{-}}{\gamma_{z}\in V_{+}}}
\bp{\wind_{rel}^{\Phi}((C;z), \R\times V_{+})-m_{z}\wind_{rel}^{\Phi}(\tilde V, \gamma_{z})} \\
\end{split}
\end{equation*}
\begin{equation*}
\begin{split}
&=\sum_{z\in\Gamma_{+}}m_{z}(\R\times \gamma_{z})\cdot \tilde V  \\
&\hskip.25in
+\sum_{\stackrel{z\in\Gamma_{+}}{\gamma_{z}\in V_{+}}}
\bp{\max\br{m_{z}\wind_{\infty}^{\Phi}(\tilde V, \gamma_{z}), \wind_{rel}^{\Phi}((C; z), \R\times V_{+})}-m_{z}\wind_{\infty}^{\Phi}(\tilde V, \gamma_{z})} \\
&\hskip.25in
+C\cdot (\R\times V_{-}) \\
&\hskip.25in
+\sum_{\stackrel{z\in\Gamma_{-}}{\gamma_{z}\in V_{-}}}
\bp{\wind_{rel}^{\Phi}((C;z), \R\times V_{-})-\min\br{m_{z}\wind_{\infty}^{\Phi}(\tilde V, \gamma_{z}),\wind_{rel}^{\Phi}((C;z), \R\times V_{-})}} \\
&\hskip.25in
+\sum_{\stackrel{z\in\Gamma_{+}}{\gamma_{z}\in V_{-}}}
\bp{m_{z}\wind_{\infty}^{\Phi}(\tilde V, \gamma_{z})-\wind_{rel}^{\Phi}((C;z), \R\times V_{-})}
\end{split}
\end{equation*}

with each of the grouped terms always nonnegative.
\end{theorem}

The nonnegativity of the terms in the above formulas allows us to establish a
convenient set of conditions which will guarantee that the  projections 
$\pi(C)$ and $\pi(\tilde V)$ of the curve and hypersurface to $M$ do not intersect.
Indeed, if $\pi(C)$ and $\pi(\tilde V)$ are disjoint then
$C_{c}$ and $\tilde V$ are disjoint for all values of $c\in\R$ and hence
the algebraic intersection number of $C_{c}$ and $\tilde V$ is zero for all values of $c\in\R$.
Since the formulas from Theorem \ref{t:intersection-computation} compute this number (for all but a finite
number of values of $c\in\R$) in terms of nonnegative quantities, we can conclude that all terms in the above formulas vanish.
We thus obtain the following corollary, which generalizes \cite[Theorem 2.4/5.12]{Siefring:2011}.

\begin{corollary}\label{c:nonintersection-conditions}
Assume that all of the hypotheses of Theorem \ref{t:intersection-computation}
hold and that $\pi(C)$ is not contained in $\pi(\tilde V)$.
Then the following are equivalent:

\begin{enumerate}
\item $\pi(C)$ and $\pi(\tilde V)$ are disjoint.

\item  All of the following hold:
\begin{enumerate}
\item
None of the asymptotic limits of $C$ intersect $\pi(\tilde V)$.

\item
$\pi(C)$ does not intersect $V=V_{+}\cup V_{-}$.

\item
For any puncture $z$ at which $C$ has asymptotic limit $\gamma^{m}$ lying in $V=V_{+}\cup V_{-}$,
$\wind^{\Phi}_{rel}((C; z), V)=m \wind_{\infty}^{\Phi}(\tilde V, \gamma)$.
\end{enumerate}
\end{enumerate}
\end{corollary}

\section{Proof of the main result}\label{s:main-proof}
We now proceed with the proof of our main result, Theorem \ref{t:main}.
We consider a manifold $M^{2n+1}$ equipped with a nondegenerate stable Hamiltonian structure
$(\lambda, \omega)$
and
an $f$-admissible fibration $(\Sigma, p, Y, f)$.
Recall that this means that
$p:M\to\Sigma$ is a smooth fibration over a closed surface with fiber diffeomorphic to $Y^{2n-1}$,
and $V:=p^{-1}(\crit(f))$ is a strong stable Hamiltonian hypersurface.
We further recall that each critical point $w$ of $f$ is assigned a sign $\sign(w)\in\br{-1, +1}$
according to whether $p_{*}:\xi_{Y_{w}}^{\perp}\to T_{w}\Sigma$ is everywhere orientation preserving or reversing.

We will let $\tl J$ denote a compatible almost complex structure on $\R\times M$ 
and will assume that $\R\times V$ has $\tl J$-invariant tangent space.
We consider an asymptotically cylindrical, $\tl J$-holomorphic map 
$\tl u=(a, u):S\setminus\Gamma\to\R\times M$ so that all asymptotic limits
$\gamma_{z}^{m_{z}}$ of $\tl u$ are contained in the binding set
$V=p^{-1}(\crit(f))$ of the foliation.
As observed in the introduction, the projected map
$v:=p\circ u:S\setminus\Gamma\to\Sigma$ admits a continuous extension
$\bar v:S\to\Sigma$ over the punctures.
The following lemma computes the degree of this map in terms of the 
intersection number and relative normal windings of $\tl u$ and any given component 
of the binding set of the foliation.

\begin{lemma}\label{l:degree-computation}
Assume the map $u:S\setminus\Gamma\to M$ does not have image
contained in $V=p^{-1}(\crit(f))$.  Then,
given a point $w\in\crit(f)$, the degree of the map
$\bar v$ defined above is given by the formula
\begin{align*}
\deg(\bar v)
&=
\sign(w)
\bp{
\tilde u\cdot (\R\times Y_{w})
-\sum_{\stackrel{z\in\Gamma^{+}}{\gamma_{z}\subset Y_{w}}}\wind^{\Phi}_{rel}((\tilde u; z), Y_{w})
+\sum_{\stackrel{z\in\Gamma^{-}}{\gamma_{z}\subset Y_{w}}}\wind^{\Phi}_{rel}((\tilde u; z), Y_{w})
} \\
&=
\sign(w)
\bp{
\tilde u*(\R\times Y_{w})
-\sum_{\stackrel{z\in\Gamma^{+}}{\gamma_{z}\subset Y_{w}}}\fl{\mu_{N}^{\Phi}(\gamma_{z}^{m_{z}})/2}
+\sum_{\stackrel{z\in\Gamma^{-}}{\gamma_{z}\subset Y_{w}}}\ceil{\mu_{N}^{\Phi}(\gamma_{z}^{m_{z}})/2}
}
\end{align*}
and thus satisfies
\begin{equation}\label{e:degree-inequality}
\begin{aligned}
\sign(w)\deg(\bar v)
&\ge
-\sum_{\stackrel{z\in\Gamma^{+}}{\gamma_{z}\subset Y_{w}}}\fl{\mu_{N}^{\Phi}(\gamma_{z}^{m_{z}})/2}
+\sum_{\stackrel{z\in\Gamma^{-}}{\gamma_{z}\subset Y_{w}}}\ceil{\mu_{N}^{\Phi}(\gamma_{z}^{m_{z}})/2} \\
&=
\sum_{\stackrel{z\in\Gamma^{\pm}}{\gamma_{z}\subset Y_{w}}}\mp\alpha_{N}^{\Phi; \mp}(\gamma_{z}^{m_{z}})
\end{aligned}
\end{equation}
with equality occurring if and only if $u$ does not intersect
$Y_{w}$ and
$\wind_{rel}^{\Phi}((\tilde u; z), Y_{w})=\alpha_{N}^{\Phi; \mp}(\gamma_{z}^{m_{z}})$ for
every $z\in\Gamma^{\pm}$ with $\gamma_{z}\subset Y_{w}$.
\end{lemma}

\begin{proof}
Given a point $w\in\crit(f)$ it's clear from the definition of the map
$\bar v$ that points $z\in S$ with $\bar v(z)=w$ coincide with intersection points of the map
$\tilde u$ with $\R\times Y_{w}$ (or equivalently, intersection points of
$u$ with $Y_{w}$) and with punctures $z\in\Gamma$ for which the periodic orbit
$\gamma_{z}^{m_{z}}$ is contained in $Y_{w}$.
Since we've observed in Corollary \ref{c:finite-intersections} this number is finite,
we can compute the degree of the map
$\bar v:S\to\Sigma$ by summing the local degree of the map
at each point in $\bar v^{-1}(w)$.

Assume for the moment that $\sign(w)=1$.
Since in this case the identification of $\xi_{Y_{w}}^{\perp}$ with $T\Sigma_{w}$ via the map
$p_{*}$ is orientation preserving, it's clear that the local index of the map
$v=p\circ u$ agrees with the intersection number of $\tilde u$ with $\R\times Y_{w}$.
Summing over all such intersection points leads to the first term in the
given formula for the degree.
To compute the local degree at a positive puncture we
first choose positive holomorphic cylindrical coordinates
$(s, t)\in [0, \infty)\times S^{1}$ on a deleted neighborhood of the puncture
$z\in\Gamma^{+}$.
Because in such a coordinate system the loop $t\mapsto (s, t)$ for fixed $s$ encircles the puncture in the
clockwise direction, the local degree of the map
$\bar v$ at $z$ can be computed by identifying a neighborhood of $w$ with
$T_{w}\Sigma$ and computing
\[
-\wind\bar v(s, t)=-\wind (p\circ u)(s, t).
\]
But considering the asymptotic representation of the normal component of the map $\tilde u$ from
Theorem \ref{t:normal-asymptotics} along with the definition of the normal relative winding, 
we have that 
\[
\wind (p\circ u)(s, t)=\wind^{\Phi}_{rel}((\tilde u; z), \R\times Y_{w})
\]
which shows the local degree of the $\bar v$ at $z$ is
\[
-\wind^{\Phi}_{rel}((\tilde u; z), \R\times Y_{w}).
\]

At negative punctures, we argue similarly but instead choose negative
cylindrical coordinates $(s, t)\in (-\infty, 0]\times S^{1}$ on a deleted neighborhood of the puncture.
Since the loop $t\mapsto (s, t)$ encircles the puncture in the counterclockwise direction,
an argument analogous to that given above tells us that 
the local degree is now given by
\[
\wind\bar v(s, t)=\wind (p\circ u)(s, t)=\wind^{\Phi}_{rel}((\tilde u; z), \R\times Y_{w}).
\]
The first line of the
claimed formula for the degree of $\bar v$ now follows in the case that
$\sign(w)=1$.
The case when $\sign(w)=-1$ is identical with the exception of the fact that the map
$p_{*}:\xi_{Y_{w}}^{\perp}\to T_{w}\Sigma$ is now orientation reversing
which introduces a factor of $-1$ into each
of the computations.

The second line in the claimed formula for the degree of $\bar v$ now follows
from the definition of the holomorphic intersection product, the definition of the asymptotic intersection numbers
and Theorem \ref{t:intersection-positivity}

Finally, the inequality \eqref{e:degree-inequality} and the claim about when equality is achieved is
an immediate consequence of local positivity of intersections and the bounds
\eqref{e:wind-rel-pos} and \eqref{e:wind-rel-neg} or, equivalently,
as a consequence of
Theorem \ref{t:intersection-positivity}.
\end{proof}

Now, in addition to the assumptions preceding Lemma \ref{l:degree-computation} we assume that $M$ is equipped
with a holomorphic foliation $\F$ compatible with the $f$-admissible fibration $(\Sigma, p, Y, f)$.
Recall this means that $\F$ is an $\R$-invariant foliation of $\R\times M$ all of whose leaves have
$\tl J$-invariant tangent spaces and are diffeomorphic to $\R\times Y$.
Moreover the leaves of the foliation fixed by the $\R$-action are precisely those contained
in $\R\times V=\R\times p^{-1}(\crit(f))$ and all other leaves project via
$\R\times M\stackrel{\pi}{\to}M$ to embeddings smoothly foliating
$M\setminus V$, and the leaves of this foliation project under
$M\stackrel{p}{\to}\Sigma$ to flow lines of the gradient of $f$ with respect to some metric on
$\Sigma$.
We observe that the assumption that leaves of the foliation project to gradient flow lines implies that
with respect to a trivialization $\Phi$ of
$\xi_{V}^{\perp}$ in the preferred homotopy class determined by the fibration, the windings
$\wind^{\Phi}_{\infty}(\tilde V, \gamma)$,
defined by \eqref{e:V-gamma-wind-1} or  \eqref{e:V-gamma-wind-2}, 
vanish for  any leaf $\tilde V$
of the foliation not fixed by the $\R$-action
and any periodic orbit $\gamma$ lying in the one of the asymptotic limits of $\tilde V$.

We now proceed with the proof of our main theorem, Theorem \ref{t:main}.
We will continue to let $(S, j, \Gamma, \tl u=(a, u))$ denote
a punctured pseudoholomoprhic curve
with all punctures limiting to periodic orbits in $V=p^{-1}(\crit(f))$, and we 
let $\bar v:S\to\Sigma$ denote the continuous extension of the map
$p\circ u:S\setminus\Gamma\to\Sigma$.
We further assume that at each puncture $z\in\Gamma$ the normal Conley--Zehnder index of the
asymptotic limits $\gamma_{z}^{m_{z}}$ satisfies
\begin{equation}\label{e:cz-condition}
\mu_N^{\Phi}(\gamma_{z}^{m_{z}})\in\br{-1, 0, 1}
\end{equation}
with $\Phi$ still denoting a symplectic trivialization of $\xi_{V}^{\perp}$ in the preferred homotopy class determined by
the $f$-admissible fibration $(\Sigma, j, p, f)$.
We then claim that:
\begin{itemize}
\item If $\sign(w)=1$ for all $w\in\crit(f)$ then the map $\bar v$ has nonnegative degree and
has degree zero
precisely when the image of the map $\tilde u$ is contained in a leaf of the foliation $\F$.

\item If $\sign(w)=-1$ for all $w\in\crit(f)$ then the map $\bar v$ has nonpositive degree and
has degree zero
precisely when the image of the map $\tilde u$ is contained in a leaf of the foliation $\F$.

\item If the map $\sign:\crit(f)\to\br{-1, 1}$ is surjective then the map $\bar v$ has degree zero and
the image of the map $\tilde u$ is contained in a leaf of the foliation $\F$.
\end{itemize}

\begin{proof}[Proof of \ref{t:main}]
The condition \eqref{e:cz-condition} implies that
\[
\alpha_{N}^{\Phi; -}(\gamma_{z}^{m_{z}})=\fl{\mu_{N}^{\Phi}(\gamma_{z}^{m_{z}})/2}\le \fl{1/2}=0
\]
for all $z\in \Gamma^{+}$ and
\[
\alpha_{N}^{\Phi; +}(\gamma_{z}^{m_{z}})=\ceil{\mu_{N}^{\Phi}(\gamma_{z}^{m_{z}})/2}\ge \ceil{-1/2}=0
\]
for all $z\in\Gamma^{-1}$, which together are equivalent to
\begin{equation}\label{e:alpha-inequality}
\mp\alpha_{N}^{\Phi; \mp}(\gamma_{z}^{m_{z}})\ge 0
\end{equation}
for all $z\in\Gamma^{\pm}$.

Assuming that $\sign(w)=1$ for all $w\in\crit(f)$, using
\eqref{e:degree-inequality} from Lemma \ref{l:degree-computation}
with \eqref{e:alpha-inequality}
shows that $\deg\bar v\ge 0$ and that
$\deg\bar v=0$ precisely when
\begin{equation}\label{e:degree-zero-condition}
\begin{gathered}
\text{$u(\Sigma\setminus\Gamma)$ and $Y_{w}$ are disjoint for all $w\in\crit(f)$, and} \\
\text{$\wind_{rel}^{\Phi}((C; z), Y_{w_{z}})=\mp\alpha^{\Phi;\mp}(\gamma_{z}^{m_{z}})=0$ for all $z\in\Gamma^{\pm}$}.
\end{gathered}
\end{equation}
Similarly if  $\sign(w)=-1$ for all $w\in\crit(f)$, the same argument shows that
$\deg\bar v\le 0$ and that
$\deg\bar v=0$ precisely when \eqref{e:degree-zero-condition} holds.

In the third case that there exist points $w_{+}$, $w_{-}\in\crit(f)$ with
$\sign(w_{+})=1$ and $\sign(w_{-})=-1$, applying
\eqref{e:degree-inequality} and \eqref{e:alpha-inequality}
with $w=w_{+}$
yields $\deg(\bar v)\ge 0$
while applying 
\eqref{e:degree-inequality} and \eqref{e:alpha-inequality} with $w=w_{-}$ yields
$\deg(\bar v)\le 0$.  We conclude that
$\deg(\bar v)=0$ which once again happens precisely when
\eqref{e:degree-zero-condition} holds.

To complete the proof of all three cases, it remains to show that
$\deg(\bar v)=0$ precisely when the image of the map $\tilde u$ is contained in a leaf of the foliation
$\F$.
Assume that $\deg(\bar v)=0$.  Since $\F$ is a foliation, and we assume that the image of $u$ is not
contained in the binding $V=\crit(f)$, there exists at least one leaf $\tilde Y$ not fixed by the $\R$-action
for which $\tilde u$ intersects $\tilde Y$.
Assume $\tilde Y$ projects via $p\circ\pi$ to a gradient flow line between critical points
$w_{1}$ and $w_{2}$ so that $\tilde Y$ is asymptotically cylindrical over $Y_{w_{1}}\cup Y_{w_{2}}$,
and recall that we've noted above that with respect to the global trivialization $\Phi$ for
$\xi_{V}^{\perp}$ we've chosen,
$\wind_{\infty}^{\Phi}(\tilde Y, \gamma)=0$ for all periodic orbits $\gamma\in Y_{w_{1}}\cup Y_{w_{2}}$.
We would now like to apply Corollary \ref{c:nonintersection-conditions} above to prove
the image of $\tilde u$ is contained in $\tilde Y$.
Assume to the contrary that the image of $\tilde u$ is not contained in $\tilde Y$.
By the assumptions that all asymptotic limits of the map $\tilde u$ are periodic orbits
in the binding $V=p^{-1}(\crit(f))$ we know that none of these limits intersect
$\pi(\tilde Y)$ since $\tilde Y$ is assumed to project to an embedding in $M\setminus V$.
Moreover, since
we have already observed that in each case, $\deg(\bar v)=0$ is true precisely when
\eqref{e:degree-zero-condition} holds, we can conclude that
$u(S\setminus\Gamma)$ does not intersect the asymptotic limit set
$Y_{w_{1}}\cup Y_{w_{2}}$ and that for each $z\in\Gamma^{\pm}$ with
$\gamma_{z}^{m_{z}}\subset Y_{w_{1}}\cup Y_{w_{2}}$
we'll have that 
$\wind^{\Phi}_{rel}((u; z), Y_{w_{1}}\cup Y_{w_{2}})=0=m_{z} \wind_{\infty}^{\Phi}(\tilde Y, \gamma_{z})$.
Corollary \ref{c:nonintersection-conditions} now lets us conclude that the image of
$\tilde u$ is disjoint from $\tilde Y$ in contradiction to the assumption that they intersect.
We thus conclude that the image of $\tilde u$ is contained in $\tilde Y$ as desired.
\end{proof}

\bibliographystyle{../../hplain5}
\bibliography{../../bibdata-new}          
\end{document}